\documentclass[11pt,a4paper]{amsart}

\usepackage{amsmath, amsthm, amssymb, graphicx, color, euscript, appendix,enumitem,microtype,hyperref,mathtools,parskip}

\newtheorem{theorem}{Theorem}
\newtheorem{consequence}{Consequence}[section]
\newtheorem{proposition}[consequence]{Proposition}
\newtheorem{lemma}[consequence]{Lemma}

\newtheorem*{thm}{Theorem}

\newtheorem*{acknowledgment}{Acknowledgment}

\theoremstyle{definition}
\newtheorem{definition}[consequence]{Definition}
\newtheorem{fact}{Fact}

\theoremstyle{remark}
\newtheorem{remark}[consequence]{Remark}

\newcommand{\R}{\mathbb{R}} 
\newcommand{\C}{\mathbb{C}} 
 
\newcommand{\N}{\mathbb{N}}
\newcommand{\D}{\mathbb{D}}

\newcommand{\cC}{{\ensuremath{\mathcal{C}}}}

\newcommand{\cO}{{\ensuremath{\mathcal{O}}}}
\newcommand{\cS}{{\ensuremath{\mathcal{S}}}}

\newcommand{\Aeul}{\EuScript{A}}
\newcommand{\Beul}{\EuScript{B}}
\newcommand{\Ceul}{\EuScript{C}}

\newcommand{\Ieul}{\EuScript{I}}
\newcommand{\Jeul}{\EuScript{J}}

\newcommand{\Seul}{\EuScript{S}}
\newcommand{\Ueul}{\EuScript{U}}
\newcommand{\Veul}{\EuScript{V}}
\newcommand{\Weul}{\EuScript{W}}

\DeclareMathOperator{\ord}{ord}
\DeclareMathOperator{\Homeo}{Homeo}
\DeclareMathOperator{\Mat}{Mat}
\DeclareMathOperator{\GL}{GL}

\begin{document}
\title{A homotopy theorem for Oka theory}
\author{Luca Studer}
\email{luca.studer@math.unibe.ch}
\begin{abstract} We prove a homotopy theorem for sheaves. Its application shortens and simplifies the 
proof of many Oka principles such as Gromov's Oka principle for elliptic submersions.
\end{abstract} 

\maketitle

\section{introduction}
\subsection{Motivation} Oka theory is the art of reducing proofs in complex geometry to purely topological statements. 
Its applications reach beyond complex geometry; for example to the study of minimal surfaces~\cite{AF}.
The power of the theory lies in the fact that there are many problems -- some of them almost a century old -- 
for which Oka theory provides the only known approaches. Examples include the following theorems.

\begin{thm}[Grauert~\cite{Grauert}]
Two complex analytic vector bundles over a Stein base which are isomorphic 
as complex topological vector bundles are complex analytically isomorphic.
\end{thm}

\begin{thm}[Forster, Ramspott~\cite{endromisbündel}] 
The ideal sheaf of a smooth complex analytic curve in a Stein manifold $X$ of 
dimension $n\geq 3$ is generated by $n-1$ holomorphic functions $X\to \C$.
\end{thm}

\begin{thm}[Gromov~\cite{elliptic bundles}]
Every continuous map from a Stein manifold $X$ to $\C^n\setminus Y$ is homotopic to a holomorphic map given that 
$Y\subset \C^n$ is an algebraic subvariety of codimension at least $2$.
\end{thm}

\begin{thm}[Leiterer~\cite{Leiterer}]
For holomorphic maps $a,b\colon X \to \Mat(n \times n, \C)$ defined on a Stein manifold $X$ 
the equation $f(x)a(x)f(x)^{-1}=b(x)$, $x \in X$ has a holomorphic solution $f\colon X \to \GL_n(\C)$ if 
there is a smooth solution of the same equation.
\end{thm}

\begin{thm}[Kutzschebauch, L\'arusson, Schwarz~\cite{KLS2}] 
A holomorphic action of a complex reductive Lie group $G$ on $\C^n$ is linearizable if there is a smooth $G$-diffeomorphism from 
$\C^n$ to a $G$-module which induces a biholomorphism on the corresponding categorical quotients, and on every reduced fiber of 
the quotient map.
\end{thm}

All known proofs of these results depend on a specific Oka principle. That is, roughly speaking, 
on a theorem which states that 
there are only topological obstructions to a complex analytic solution of an associated problem. 
In concrete terms, Grauert's result is proved using the Oka principle for principal $G$-bundles~\cite{Cartan, Grauert}. All others 
depend on extensions of Grauert's work, namely on 
the Oka principle for admissible pairs of sheaves~\cite{Forster und Ramspott} in the case 
of Forster and Ramspott's and Leiterer's results, on the Oka principle for elliptic 
submersions~\cite{FP, elliptic bundles} in the case of Gromov's result, and on 
the Oka principle for equivariant isomorphisms~\cite{KLS} in the case of the 
result due to Kutzschebauch, L\'arusson and Schwarz. The first three theorems can be proved 
alternatively with Forstneri\v c's Oka principles for stratified fiber bundles~\cite{stratified Oka principle}, a generalization of Gromov's work to stratified settings 
including more general fibers and possibly non-smooth base spaces.
Complete proofs of these powerful tools fill a book. However, a careful study of the literature 
reveals that all proofs of the cited work can be 
divided into a rather analytic first part and a purely topological second part; and that the topological part can 
be formulated very generally, thus providing a reduction of the proofs to the analytic key difficulties. 
This general topological statement is Theorem~\ref{t1} from the present text: 
its assumptions state which key properties one has to show in the first part of the proof of an Oka principle, 
its conclusion is an Oka principle. Theorem~\ref{t1} extends Gromov's \textit{homomorphism theorem} 
from~\cite{partial differential relations} so that it applies in complex 
analytic settings and carries out ideas sketched in~\cite{elliptic bundles}. Its proof builds on ideas of 
Gromov~\cite{partial differential relations, elliptic bundles} and on the work of Forstneri\v c and Prezelj~\cite{FP}, who have carried out 
many steps of the proof of Theorem~\ref{t1} in the special case of elliptic submersions. References 
to the sections in~\cite{Cartan, Forster und Ramspott, FP, elliptic bundles, KLS} providing 
the analytic key ingredients from the assumptions in Theorem~\ref{t1} are given in the appendix. 

\subsection{Results} To formulate Theorem~\ref{t1} let us recall some basic 
notions of complex geometry. Let $X$ be a reduced complex space. A compact $C\subset X$ is called a \textit{Stein compact} 
if it admits a basis of Stein neighborhoods in $X$. If $C\subset B$ for some $B\subset X$, 
then the compact $C$ is called \textit{$\cO(B)$-convex} if $C=\{p \in B \colon  |f(p)| \leq \max_{x \in C} |f(x)|, \ f \in \cO(B)\}$. 
Here, $\cO(B)$ denotes the set of all holomorphic functions defined on unspecified 
neighborhoods of $B$.

\begin{definition}
\label{C-pair}
Let $X$ be a complex space and let $A,B \subset X$. The ordered pair $(A,B)$ is a \textit{$\cC$-pair} if 
\begin{enumerate}
\item $A,B, A \cap B$ and $A \cup B$ are Stein compacts, 
\item $A \cap B$ is $\cO(B)$-convex, and
\item $\overline{A \setminus B} \cap \overline{B \setminus A}= \emptyset$.
\end{enumerate}
\end{definition}

A convenient notion to formulate any question arising naturally in Oka theory 
is the notion of a sheaf of topological spaces. A \textit{sheaf of topological spaces} $\Phi$ on a topological 
space $X$ is a sheaf $U\mapsto \Phi(U)$, $U\subset X$ open, whose sets of local sections are topological spaces, whose 
restrictions are continuous, and which is well-behaved in the sense that for the closed unit ball $\D\subset \R^n$ of any real dimension $n\geq 1$ the presheaf 
$U \mapsto \Phi^\D(U)$ is in fact a sheaf. Here $\Phi^\D(U)$ denotes the set of continuous maps 
$\D\to \Phi(U)$. It is more common to ask for $U\to \Phi^Y(U)$ being a sheaf for any topological space $Y$. 
However, we stick to the given definition since it reflects what we need in this text.
A sheaf of topological spaces is said to be \textit{metric} 
if every set of local sections is a metric space, and a metric sheaf is said to be 
\textit{complete} if every set of local sections is a complete metric space. 
We say that $\Phi \hookrightarrow \Psi$ is an \textit{inclusion of sheaves of 
topological spaces} if $\Phi$ is a subsheaf of the sheaf of topological spaces $\Psi$ whose sets of local sections 
are endowed with the subspace topology of the sets of local sections of $\Psi$. 

We now define those properties for sheaves of topological spaces which the first part of 
a proof of an Oka principle is usually concerned with (Definition~\ref{d1} and \ref{weakly flexible}). 
In the following $s,t$ denote numbers in the unit interval and by using them 
as subscripts we always indicate a homotopy in the common sense. 
We use the convention that a singleton (with empty boundary) is the 
closed unit ball of dimension $n=0$. For a subset $A$ of $X$ we denote by $A^\circ$ its interior. 
A restriction $\Psi(U)\to \Psi(V)$ of a sheaf $\Psi$ will be denoted by $r_V$ if there is no ambiguity regarding the domain. 

\begin{definition}
\label{d1}
An inclusion of sheaves of topological spaces $\Phi \hookrightarrow \Psi$ over a topological 
space $X$ is a \textit{local weak homotopy equivalence} if for every point 
$p \in X$, every open neighborhood $U$ of $p$ and every continuous map 
$f$ from the closed unit ball $\D \subset \R^n$ of any real dimension $n\geq 0$ to 
$\Psi(U)$ with $f(\partial \D)\subset \Phi(U)$, there is a neighborhood 
$p \in V \subset U$ and a homotopy $f_t\colon \D \to \Psi(V)$ such that $f_0=r_V \circ f$, 
$f_t|\partial \D$ is independent of $t$ and $f_1$ has values in $\Phi(V)$.
\end{definition}

\begin{remark}
Local weak homotopy equivalences were introduced in \cite{partial differential relations}. 
Heuristically a local weak homotopy equivalence of 
sheaves $\Phi \hookrightarrow \Psi$ is a weak homotopy equivalence at the level of stalks, 
meaning that $\Phi_p \hookrightarrow \Psi_p$ is a weak homotopy 
equivalence for every $p \in X$. However there is no suitable topology 
on the stalks to make this heuristic a precise statement.
\end{remark}

\begin{definition}
\label{weakly flexible}
Let $\Phi$ be a sheaf of topological spaces on $X$ and let $A,B \subset X$ be compact. 
The ordered pair $(A,B)$ is called \textit{weakly flexible} for $\Phi$ if the following holds. 
Given open neighborhoods $U,V$ resp.~$W$ of $A,B$ 
resp.~$A \cap B$ and a triple of maps $a\colon \D \to \Phi(U)$, $b\colon  \D \to \Phi(V)$ 
and $c_s\colon  \D \to \Phi(W)$ such that $c_0=r_W \circ a$, $c_1=r_W \circ b$ 
and $c_s|\partial \D$ is independent of $s$, there are smaller neighborhoods 
$A\subset U' \subset U$, $B \subset V' \subset V$ resp.~$A \cap B \subset W'\subset W$ 
and homotopies $a_t\colon  \D \to \Phi(U')$, $b_t\colon  \D \to \Phi(V')$ and $c_{s,t}\colon  \D \to \Phi(W')$ 
with $a_0=r_{U'} \circ a$, $b_0=r_{V'} \circ b$ and $c_{s,0}=r_{W'} \circ c_s$ such that 
\begin{enumerate}
\item $c_{0,t}=r_{W'} \circ a_t$ and $c_{1,t}=r_{W'} \circ b_t$, 
\item the restrictions $a_t|\partial \D$, $b_t|\partial \D$ and $c_{s,t}|\partial \D$ are independent of $t$, 
\item $c_{s,1}$ is independent of $s$, and 
\item $r_{A^\circ} \circ a_t$ is in a prescribed neighborhood of $r_{A^\circ} \circ a_0\colon \D \to \Phi(A^\circ)$ with respect to the compact open topology for all $t$.
\end{enumerate}
If $(A,B)$ is a weakly flexible pair for $\Phi$ such that the homotopy $a_t$ from the 
conclusion of the definition can be chosen to satisfy $r_{A^\circ} \circ a_t=r_{A^\circ} \circ a_0$ for all $t$, then 
$(A,B)$ is called an \textit{ordered flexible pair}.
\end{definition}

\begin{remark}
Showing that $\cC$-pairs $(A,B)$ are weakly flexible for a given sheaf $\Phi$ 
is often the main work in the proof of an Oka principle. 
\end{remark}

\begin{theorem}
\label{t1}
Let $X$ be a second countable reduced Stein space and let $\Phi \hookrightarrow \Psi$ be a local weak 
homotopy equivalence of sheaves of topological spaces on $X$. 
Assume that one of the two following statements holds: 
\begin{enumerate}
\item $\Phi$ is complete metric and every point $p \in X$ has a neighborhood $U$ such that every $\cC$-pair $(A,B)$ with $B\subset U$ is weakly flexible for $\Phi$, 
\item every point $p \in X$ has a neighborhood $U$ such that every $\cC$-pair $(A,B)$ with $B\subset U$ is ordered flexible for $\Phi$.
\end{enumerate}
Then $\Phi(X)\not= \emptyset$ if and only if $\Psi(X)\not = \emptyset$. Moreover, if $\Psi$ 
is likewise in either of the two classes of sheaves, then $\Phi(X) \hookrightarrow \Psi(X)$ is a weak homotopy equivalence.
\end{theorem}

\begin{remark}
\label{analytic}
Since analytic continuation is unique, assumption (2) from Theorem~\ref{t1} is too strong to be satisfied 
by a sheaf $\Phi$ of analytic maps, even in the most basic case $\Phi=\cO_{\C}$ -- the sheaf of holomorphic functions 
in one complex variable. In particular assumption (1) is the property that one tries to show if $\Phi$ is a complex analytic sheaf.
\end{remark}

\begin{remark}
\label{remark ordered flex}
Many sheaves of interest in Oka theory are complete metric. However, there are exceptions. 
Such exceptions -- discussed briefly in the subsection \textit{ordered flexibility} of the appendix -- 
motivated to include assumption (2) as an alternative to assumption (1) in Theorem~\ref{t1}.
\end{remark}

As we will see Theorem~\ref{t1} follows from a more abstract homotopy theorem, namely Theorem~\ref{t3}. 
To state it we need

\begin{definition}
\label{weak flexibility}
A \textit{weakly flexible string} for $\Phi$ of length $n\geq 2$ is recursively defined 
as a finite sequence $(A_1, A_2,A_3, \ldots, A_n)$ of compacts of $X$ such that 
\begin{enumerate}
\item $(A_1 \cup \cdots \cup A_{n-1} , A_n)$ is a weakly flexible pair for $\Phi$, and, if $n\geq 3$, then
\item $(A_1, \ldots, A_{n-1})$ and $(A_1\cap A_n, \ldots , A_{n-1}\cap A_n)$ are weakly flexible strings for~$\Phi$. 
\end{enumerate}
A \textit{weakly flexible cover} for $\Phi$ is a locally finite cover $(A_1, A_2, A_3, \ldots)$ of $X$ 
such that for every $n\in \N$  $(A_1, A_2, \ldots, A_n)$ is a weakly flexible string for $\Phi$. 
We say that $\Phi$ is 
\textit{weakly flexible} if every open cover of $X$ can be refined by a weakly flexible cover for $\Phi$. 
Similarly we define \textit{ordered flexible} sheaves. That is, \textit{ordered flexible strings, covers} 
and \textit{sheaves} are defined by substituting each occurrence of \textit{weakly} by \textit{ordered} 
in the definition of weakly flexible strings, covers and sheaves.
\end{definition}

In the following the \textit{dimension} of a topological space always means the covering dimension.

\begin{theorem}
\label{t3}
Let $X$ be a paracompact Hausdorff space that has an exhaustion by finite dimensional 
compact subsets and let $\Phi \hookrightarrow \Psi$ be a local weak 
homotopy equivalence of sheaves of topological spaces on $X$. 
Assume that $\Phi$ is either 
\begin{enumerate}
\item complete metric and weakly flexible, or 
\item ordered flexible.
\end{enumerate}
Then $\Phi(X)\not= \emptyset$ if and only if $\Psi(X)\not = \emptyset$. Moreover, if $\Psi$ 
is likewise in either of the two classes of sheaves, then $\Phi(X) \hookrightarrow \Psi(X)$ is a weak homotopy equivalence.
\end{theorem}

\begin{remark}
\label{f/of/wf}
For a sheaf of topological spaces $\Phi$ one can show that flexibility in the sense of Gromov (for a definition see~\cite{partial differential relations}) 
implies ordered flexibility, that ordered flexibility implies weak flexibility, and that both converse implications are wrong.
\end{remark}

\begin{remark}
Gromov's \textit{homomorphism theorem}~\cite{partial differential relations}, p.~77 says that if $\Phi \hookrightarrow \Psi$ is 
a local weak homotopy equivalence and both $\Phi$ and $\Psi$ are flexible, then $\Phi(X) \hookrightarrow \Psi(X)$ is a weak homotopy equivalence. 
Therefore Theorem~\ref{t3} extends Gromov's result and it follows from Remark~\ref{analytic} and~\ref{f/of/wf} that 
this extension is necessary when working in analytic settings.
\end{remark}

\newpage

\begin{acknowledgment}
\normalfont
I would like to thank Frank Kutzschebauch for suggesting the topic and many helpful discussions. 
Moreover I would like to thank Finnur L\'arusson and Gerald Schwarz for numerous valuable comments on a preprint. 
I am also very thankful for stimulating discussions with Jasna Prezelj and Franc Forstneri\v c. 
\end{acknowledgment}

\section{proofs}
In the first two subsections we define relevant notions and establish preliminaries. In the third subsection the so-called initial complex 
is constructed (Proposition~\ref{comp}), and in the fourth subsection we glue families of local sections to a global one (Proposition~\ref{p2}). 
This part follows at some points closely Prezelj~\cite{Prezelj thesis}, where the corresponding work is done in the special case of sheaves of 
sections of elliptic submersions, see also~\cite{Francs book,FP}. 
In the last subsection Theorem~\ref{t1} and~\ref{t3} are proved.

\subsection{Parametric sheaves}
For topological spaces $Y$ and $Z$, $Z^Y$ denotes the space of continuous maps $Y \to Z$ equipped with the compact open topology. 

\begin{lemma}
\label{correspondence}
Let $P, Y, Z$ be topological spaces where $Y$ is locally compact Hausdorff. Then a map $f\colon P \times Y \to Z$ is continuous if 
and only if the map $\tilde f\colon  P \to Z^Y$ defined by $\tilde f(p)\coloneqq f(p, \cdot)$ is continuous.
\end{lemma}

\begin{proof}
This is well-known and can be found in textbooks.
\end{proof}

Let $\Phi$ be a sheaf of topological spaces on a topological space $X$ and let 
$\D\subset \R^n$ be the closed unit ball of any real dimension $n\geq 0$. Then 
$U\mapsto \Phi^\D(U)$ is a sheaf by definition. We equip the sets of local sections 
with the compact open topology and note that the restrictions are continuous.

\begin{lemma}
\label{ps}
The sheaf $\Phi^\D$ is a sheaf of topological spaces.
\end{lemma}

\begin{proof}
We have to show that $U \mapsto (\Phi(U)^\D)^{\D'}$ is a sheaf for the closed unit ball $\D'\subset \R^m$ of any real 
dimension $m\geq 0$. By applying Lemma~\ref{correspondence} to $P=\D'$, $Y=\D$ and $Z=\Phi(U)$ we get a natural correspondence of 
$(\Phi^\D)^{\D'}(U)$ and $\Phi^{\D' \times \D}(U)$. This correspondence 
yields an isomorphism of presheaves 
$$(\Phi^\D)^{\D'} \cong \Phi^{\D' \times \D}.$$ Since $\Phi$ is a sheaf of topological spaces and $\D' \times \D$ is homeomorphic to the closed unit ball 
of dimension $n+m$, 
$\Phi^{\D' \times \D}$ is a sheaf, hence so is $(\Phi^\D)^{\D'} $. 
\end{proof} 

For a continuous map $\alpha\colon  \partial \D \to \Phi(X)$ defined on the boundary of the closed unit ball $\D\subset \R^n$ 
set $\Phi_\alpha(U)=\{f \in \Phi^\D(U) \colon  f|\partial \D= r_U \circ \alpha\}$ 
for an open $U\subset X$. Note that these sets of local sections define a subsheaf $\Phi_\alpha$ of $\Phi^\D$ and 
inherit the structure of a sheaf of topological spaces by taking the subspace topology on $\Phi_\alpha(U)\subset \Phi^\D(U)$ on each set 
of local sections. 

We give an alternative definition of a weakly flexible pair (Lemma~\ref{weakly flexible pair 2}). 
Let $\Phi(U,V)$ denote the path product of $\Phi(U)$ and $\Phi(V)$, that is
$$\{(a,b,c)\in \Phi(U) \times \Phi(V) \times \Phi^{[0,1]}(U\cap V)\colon  r_{U \cap V}(a)=c(0), \ r_{U\cap V}(b)=c(1)\},$$ 
and let us equip $\Phi(U,V)$ with the subspace topology induced 
by the product topology on $\Phi(U) \times \Phi(V) \times \Phi^{[0,1]}(U\cap V)$.
Moreover we identify the image of the inclusion 
$$\Phi(U\cup V) \hookrightarrow \Phi(U,V), \ \ f\mapsto (r_U(f), r_V(f), r_{U\cap V}(f))$$ 
with $\Phi(U \cup V)$. For $U'\subset U$ and $V' \subset V$ the 
map $r_{U',V'}\colon \Phi(U,V) \to \Phi(U',V')$ is defined as 
$r_{U',V'}(f_1,f_2,f_3)=(r_{U'}(f_1),r_{V'}(f_2), r_{U'\cap V'}(f_3))$. Moreover let $p\colon  \Phi(U,V) \to \Phi(U)$ be the projection 
onto the first component. The following Lemma is straightforward. We leave the proof to the reader. 

\begin{lemma}
\label{weakly flexible pair 2}
Let $\Phi$ be a sheaf of topological spaces on $X$ and let $A,B \subset X$ be compact. 
The pair $(A,B)$ is weakly flexible for $\Phi$ if and only if the following holds. Given 
open neighborhoods $U$ and $V$ of $A$ and $B$ and a map $f\colon  \D \to \Phi(U,V)$ 
with $f(\partial \D)\subset \Phi(U \cup V)$, there are smaller neighborhoods $U'$ and $V'$ of $A$ and $B$ and a 
homotopy $g_t\colon I \times \D \to \Phi(U',V')$ such that $g_0=r_{U',V'}\circ f$, $g_1$ 
has values in $\Phi(U'\cup V')$, $g_t|\partial \D$ is independent of 
$t$ and $r_{A^\circ} \circ p \circ g_t$ stays in a prescribed neighborhood of $r_{A^\circ} \circ p \circ g_0$ for all $t \in [0,1]$.
\end{lemma}

\begin{lemma}
\label{parametric sheaves}
Let $\Phi$ be a sheaf of topological spaces on a topological space $X$ and let $\alpha\colon  \partial \D \to \Phi(X)$ be continuous, 
where $\D\subset \R^n$ is the closed unit ball of any real dimension $n\geq0$. 
Then, if any of the following properties holds for $\Theta=\Phi$, the same property holds likewise for $\Theta=\Phi_\alpha$: 
\begin{enumerate}
\item $\Theta$ is a sheaf of topological spaces, 
\item $\Theta$ has the structure of a complete metric sheaf, 
\item A family of compacts $(A_1,A_2, \ldots, A_n)$ is weakly flexible for $\Theta$,
\item A family of compacts $(A_1,A_2, \ldots, A_n)$ is ordered flexible for $\Theta$,
\item $\Theta$ is weakly flexible, and
\item $\Theta$ is ordered flexible.
\end{enumerate}
Moreover, if $\Phi\hookrightarrow \Psi$ is a local weak homotopy equivalence of sheaves, then so is $\Phi_\alpha \hookrightarrow \Psi_\alpha$.
\end{lemma}

\begin{proof}
It has been discussed already that $\Phi_\alpha$ is a sheaf of topological spaces if $\Phi$ is. 
If $\Phi(U)$ is a metric space with metric $d$, then $$d_\alpha(f,g)\coloneqq \sup\{d(f(p),g(p))\colon  p \in \D\}$$ defines a metric 
on $\Phi_\alpha(U)$ which agrees with the compact open topology. Moreover if $\Phi(U)$ is complete, then so is $\Phi_\alpha(U)$ 
since limits of continuous maps from $\D$ to complete spaces are continuous. For (3) and (5) it suffices to show that any 
weakly flexible pair $(A,B)$ for $\Phi$ is weakly flexible for $\Phi_\alpha$. We use the alternative 
definition of a weakly flexible pair given in Lemma~\ref{weakly flexible pair 2}. 
Let $f\colon \D' \mapsto \Phi_\alpha(U,V)$ be a map as in this assumption 
(see Lemma~\ref{weakly flexible pair 2}), where $\D' \subset \R^m$ denotes the closed unit ball 
of dimension $m\geq 0$. By Lemma~\ref{correspondence} $f$ can be canonically identified 
with a continuous map $g\colon  \D' \times \D \to \Phi(U, V)$ with 
$$g(\partial (\D' \times \D))=g(\partial \D' \times \D \cup  \D' \times \partial \D)\subset \Phi(U\cup V).$$ 
Since $(A,B)$ is a weakly flexible pair for $\Phi$ and $\D' \times \D$ is a disc of dimension $m+n$, 
there are neighborhoods $U'$ and $V'$ of $A$ resp.~$B$ and 
a homotopy $g_t$ connecting $r_{U',V'}\circ g$ to some $g_1$ with values in $\Phi(U' \cup V')$, 
$g_t$ is independent of $t$ when restricted to the boundary, and $r_{A^\circ} \circ p \circ g_t$ stays in a 
prescribed neighborhood of $r_{A^\circ} \circ p \circ g_0$. Since $g_t$ is independent of $t$ when restricted to 
$\D' \times \partial \D$, $g_t$ is then again canonically identified with a homotopy 
$$f_t\colon  \D' \to \Phi_\alpha (U',V')$$ satisfying $f_0=r_{U',V'} \circ f$. Since $g_t$ is also independent of $t$ 
when restricted to $\partial \D' \times  \D$, $f_t|\partial \D'$ is independent of $t$ as well. 
It follows from $g_1$ having values in $\Phi(U'\cup V')$ that $f_1$ has values in $\Phi_{\alpha}(U' \cup V')$. 
The fact that $r_{A^\circ} \circ p \circ f_t$ 
approximates $r_{A^\circ} \circ p \circ f_0$ as well as we desire follows immediately from the corresponding 
approximation property of $g_t$. This shows that $(A,B)$ is 
weakly flexible for $\Phi_\alpha$ and concludes the proof of (3) and (5). 
The proof that an ordered flexible pair $(A,B)$ for $\Phi$ is also ordered flexible for $\Phi_\alpha$ is similar (but a little simpler).
Statement (4) and (6) follow immediately. 
Also the proof that $\Phi_\alpha \hookrightarrow \Psi_\alpha$ is a weak homotopy equivalence if $\Phi \hookrightarrow \Psi$ is 
can be obtained in the same manner as the proof that weakly flexible pairs for $\Phi$ are weakly flexible for $\Phi_\alpha$.
We leave the details to the reader. 
\end{proof}

\subsection{Complexes} In this subsection we define the notion of a complex. 
Let $\Aeul=\{A_i\}_{i \in \Ieul}$ be a point-finite family of subsets of some topological space X (that is, 
every element of $X$ is contained in at most finitely many elements of $\Aeul$). 
The \textit{nerve} $N(\Aeul)$ of $\Aeul$ is defined as the set of (finite) subsets 
$I$ of the index set $\Ieul$ for which the intersection $A_I\coloneqq \cap_{i \in I} A_i$ is non-empty. 
Let $E$ be the real vector space spanned by linearly independent vectors $\{e_i\}_{i\in \Ieul}$ 
and let us topologize $E$ by the final topology induced by injective linear maps 
$\R^n \to E$, $n\in \N$. Clearly any finite dimensional subspace of $E$ carries the 
natural topology. For $I \in N(\Aeul)$, denote by $|I|$ the simplex defined by the 
convex hull of $\{e_i\colon  i \in I\} \subset E$. Note that the union $N|\Aeul|$ of all 
$|I|$ with $I \in N(\Aeul)$ is the geometric realization of the abstract simplicial complex $N(\Aeul)$. 
We also set $N_k(\Aeul)\coloneqq \{I \in N(\Aeul)\colon  \dim |I|\leq k\}$ and $N_k|\Aeul|$ the 
corresponding geometric realization for $k\geq 0$. A family $\Ueul=\{U_i\}_{i \in \Ieul}$ is 
a \textit{neighborhood} of $\Aeul=\{A_i\}_{i \in \Ieul}$ if $U_i$ is a neighborhood of $A_i$ for every $i \in \Ieul$. 
$\Ueul$ is called \textit{faithful} if its nerve is equal to the one of $\Aeul$. 
\begin{lemma}
\label{faithful neighborhood}
Let $X$ be normal and $\Aeul$ a countable locally finite closed cover of $X$. 
Then $\Aeul$ has a closed faithful neighborhood. 
\end{lemma}

\begin{proof}
Let $A \in \Aeul=\{A_i\}_{i \in \N}$ be fixed. 
Recall that the union of the elements of a locally finite closed family is closed. 
Since $\Aeul$ is locally finite, 
$\{A_I\colon  I \in N(\Aeul)\}$ is locally finite too. In particular 
$$U\coloneqq X\setminus \bigcup_{I \in N(\Aeul), \atop A\cap A_I=\emptyset} A_I,$$ 
is an open neighborhood of $A$. Note that $\{U\} \cup \Aeul\setminus \{A\}$ has the same nerve as $\Aeul$. 
If $B\subset U$ is a closed neighborhood of $A$, then $\{B\}\cup \Aeul\setminus \{A\}$ is a locally finite closed cover 
with the same nerve as $\Aeul$. Replacing inductively $A_n$, $n \in \N$ in $\{B_1,B_2, \ldots, B_{n-1}, A_n, A_{n+1}, \ldots\}$ 
by such a neighborhood $B_n$ yields 
for $n \to \infty$ a closed faithful neighborhood $\Beul$ of $\Aeul$.
\end{proof}

For a sheaf $\Psi$ we say a map has values in $\Psi$ when we actually mean a map with 
values in the disjoint union of all sets of local sections of $\Psi$. 
For an open set $V\subset X$, $\Psi_V$ denotes the disjoint union  of those $\Psi(U)$ with $V\subset U$. 
Then $r_V\colon  \Psi_V \to \Psi(V)$ denotes the map that is given by the restriction morphisms $\Psi(U) \to \Psi(V)$ of $\Psi$ for $V \subset U$ open. 
We call $r_V$ a \textit{restriction} and note that the new use of the symbol $r_V$ generalizes our earlier use consistently.

\begin{definition}
Let $\Aeul$ be a point-finite family of $X$. 
A \textit{homotopy of complexes} over $\Aeul$ (with respect to the sheaf $\Psi$) 
is a map $$f\colon  [0,1] \times N|\Aeul|  \to \Psi$$ with the property that there is an open neighborhood $\Ueul$ of 
$\Aeul$ such that $f([0,1] \times |I|) \subset \Psi_{U_I}$ and $r_{U_I} \circ f|[0,1] \times |I|$ is continuous for 
all $I \in N(\Aeul)$. A \textit{complex} over $\Aeul$ is a map $N|\Aeul| \to \Psi$ with 
the properties obtained by fixing $t\in [0,1]$ in the definition of a homotopy of complexes.
\end{definition}
If we replace the domain $N|\Aeul|$ by $N_k|\Aeul|$ in the definition of a complex we call the resulting map a \textit{$k$-skeleton} instead. 
If the family $\Aeul$ or the sheaf $\Psi$ corresponding to a complex $f\colon N|\Aeul| \to \Psi$ are clear from the context we will sometimes omit to 
mention these.
\begin{remark}
\label{subcover}
If a point-finite cover $\Aeul=\{A_i\}_{i \in \Ieul}$ refines a point-finite cover 
$\Ueul=\{U_j\}_{j \in \Jeul}$, then a complex $f\colon N|\Ueul| \to \Psi$ induces a complex $g\colon N|\Aeul| \to \Psi$. 
Explicitly this can be seen by setting $g=f \circ \lambda$ where $\lambda\colon N|\Aeul| \to N|\Ueul|$ 
is the restriction of a linear map given by $e_i \mapsto e_{j(i)}$ for some $j(i)$ with $A_i\subset U_{j(i)}$.
\end{remark}

\begin{definition}
Two complexes $f,g\colon  N|\Aeul| \to \Psi$ are called 
\textit{equivalent} if there is a neighborhood $\Ueul=\{U_i\}_{i \in \Ieul}$ of $\Aeul=\{A_i\}_{i \in \Ieul}$ such that 
$r_{U_I} \circ f| |I| = r_{U_I} \circ g||I|$
for $I \in N(\Aeul)$.
\end{definition}
It is straightforward to see that equivalence of complexes over $\Aeul$ is an 
equivalence relation. Since in this text we care about complexes 
only up to equivalence, we write $f=g$ if $f$ and $g$ are equivalent. 
Note that we can compose homotopies of complexes $f_t$ and $g_t$ (over $\Aeul$) 
if $f_1$ is equivalent to $g_0$ 
by setting $h_{t/2}=f_t$ for $t\in [0,1]$ and 
$h_{t/2+1/2}=g_t$ for $t \in (0,1]$.

\subsection{The initial complex} In this subsection we prove 

\begin{proposition}
\label{comp}
Let $X$ be a paracompact Hausdorff space that has an exhaustion by finite dimensional closed subsets. 
Let $\Phi \hookrightarrow \Psi$ be a local weak homotopy equivalence 
of topological sheaves on $X$ and let $f \in \Psi(X)$. Then there is a homotopy 
of complexes $h$ with values in $\Psi$ connecting the complex $h_0$ 
induced by $f$ to a complex $h_1$ with values in $\Phi$. 
\end{proposition}

The proof of Proposition~\ref{comp} depends on some topological properties of $X$. 
For a topological space $X$ and an open cover $\{V_j\}_{j \in \Jeul}$, 
a \textit{shrinking} of $\{V_j\}_{j \in \Jeul}$ is an open cover 
$\{S_j\}_{j \in \Jeul}$ (indexed by the same set $\Jeul$) such that $\overline S_j\subset V_j$ for all $j \in \Jeul$.

\begin{remark}
\label{l1}
It is well-known that in a paracompact Hausdorff space any open cover has a shrinking.
\end{remark}

In the following we set for a family of subset $\Ueul$ of $X$ and $A\subset X$, $\Ueul_A\coloneqq \{U \in \Ueul\colon  A\cap U \not = \emptyset\}$. 
Moreover $\ord \Ueul = \sup_{x \in X} |\Ueul_{\{x\}}| \in \N\cup \{\infty\}$ denotes the order of $\Ueul$. We use the conventions $\ord \emptyset = 0$. 

\begin{lemma}
\label{cover0}
Every open cover $\Ueul$ of a finite dimensional paracompact Hausdorff space $B$ 
has an locally finite open refinement of order at most $\dim B +1$.
\end{lemma}

\begin{proof}
Pick a locally finite open refinement $\Veul$ of $\Ueul$ and then an open 
refinement $\Weul$ of $\Veul$ of order at most $\dim B +1$. 
Fix for every element $W \in \Weul$ an element $V(W) \in \Veul$ with $W \subset V(W)$, set 
$$V'\coloneqq \bigcup_{V(W)=V} W, \ \ V \in \Veul$$ and 
$\Veul'\coloneqq \{V'\colon  V \in \Veul\}$. The open family $\Veul'$ covers $B$ since $\Weul$ covers $B$. 
It is locally finite since $\Veul$ is locally finite, and it is easy to see that $\ord \Veul' \leq \ord \Weul \leq \dim B +1$. 
This finishes the proof. 
\end{proof}

\begin{lemma}
\label{cover1}
Let $X$ be a paracompact Hausdorff space, $A,B\subset X$ closed such that $A\subset B^\circ$ and $\dim B< \infty$ and let $W\subset A$ open. 
Then any open family $\Ueul$ of $X$ which covers $X\setminus W$ can be refined by a locally finite open family $\Aeul$ which covers $X\setminus W$ such that 
$\Aeul_A$ is of order at most $\dim B +1$.
\end{lemma}

\begin{proof}
Let $\Ueul$ be an open family which covers $X\setminus W$. Since the union 
of two locally finite families remains locally finite, it suffice to find (1) a 
locally finite open family which refines $\Ueul$, covers $X\setminus B^\circ$ and whose 
elements do not intersect $A$ and (2) a locally finite open family of order at 
most $\dim B + 1$ which refines $\Ueul$ and covers $B^\circ \setminus W$. For (1) we 
can refine $\Ueul \cup \{W\}$ by a locally finite open cover of $X$ and intersect the elements of 
the resulting family with the open set $X\setminus A$. For (2) apply Lemma~\ref{cover0} to the 
open cover $\{W\} \cup \{U\cap B\colon  U \in \Ueul\}$ of $B$, remove from the 
obtained family the elements which are contained in $W$ and intersect the remaining 
elements with $B^\circ$. The resulting family does the job. 
\end{proof}

\begin{lemma}
\label{neighborhood1}
Let $X$ be a topological space, $\Veul=\{V_i\}_{i \in \Ieul}$ a locally finite open 
cover and $\Seul=\{S_i\}_{i \in \Ieul}$ a shrinking of  $\{V_i\}_{i \in \Ieul}$. 
Then for every point $x \in X$ there is an open neighborhood $W$ of $x$ with 
the property $W \subset V_I$ whenever $W \cap S_I \not = \emptyset$.
\end{lemma}

\begin{proof}
Since $\Veul$ is locally finite and $\Seul$ is a shrinking of $\Veul$ there is 
a neighborhood $U$ of $x$ that meets 
only finitely many $\overline S_i$, $i \in \Ieul$. Therefore
$$W\coloneqq  U\cap \bigcap_{x \in \overline S_i} V_i \cap \bigcap_{x \not \in \overline S_i, \atop U \cap \overline S_i \not = \emptyset} X \setminus \overline S_i$$ 
is a finite intersection of open sets and hence an open neighborhood of $x$. One can check that $W$ does the job.
\end{proof}

The main work of the proof of Proposition~\ref{comp} is done in Lemma~\ref{inductive step}, 
in whose proof a suitable homotopy of $k+1$-skeletons is constructed from a given homotopy of $k$-skeletons. 
\begin{lemma}
\label{inductive step}
Let $A_0= A_1=\emptyset$, let $A_2\subset A_3 \subset A_4 \subset \cdots$ be an exhaustion 
of finite dimensional closed subsets of $X$ and let $\Phi \hookrightarrow \Psi$ 
be a local weak homotopy equivalence of sheaves on $X$. Let $k\geq 0$, $\Veul$ a locally 
finite open cover of $X$ such that each element of $\Veul$ is 
contained in $A_{i+1}\setminus A_{i-1}$ for suitable $i \in \N$, and let $n \in \N$ be 
some fixed number such that $\ord \Veul_{A_n}\leq k$. Moreover let 
$$g\colon [0,1] \times N_k|\Veul| \to \Psi$$ be a homotopy of $k$-skeletons which connects the 
sectionally constant complex given by $f$ to a $k$-skeleton with values in $\Phi$.
Then there is a locally finite open cover $\Ueul$ of $X$ refining $\Veul$ and a $k+1$-skeleton 
$$h\colon [0,1] \times N_{k+1}|\Ueul| \to \Psi$$ connecting the sectionally constant 
complex given by $f$ to a $k+1$-skeleton with values in $\Phi$
satisfying the following properties: The elements of $\Ueul$ are contained in $A_{i+1}\setminus A_{i-1}$ for suitable $i \in \N$, 
$\ord{\Ueul_{A_n}} \leq \ord{\Veul_{A_n}}$, $\ord{\Ueul_{A_{n+1}}}\leq \ord \Veul_{A_{n}}+\dim A_{n+2}+1$, 
$\Ueul_{A_{n-1}}=\Veul_{A_{n-1}}$ and 
\begin{align*}
g|[0,1]\times N|\Veul_{A_{n-1}}|=h|[0,1]\times N|\Ueul_{A_{n-1}}|.
\end{align*}
\end{lemma}

\begin{proof}
Take a shrinking $\Seul$ of $\Veul=\{V_j\}_{j \in \Jeul}$ and set 
$$W\coloneqq \bigcup_{S \in \Seul_{A_n}} S.$$ 
We have $A_n\subset W$ since $\Seul$ covers $X$, and $W\subset A_{n+1}$ 
since $\Seul$ is a shrinking of $\Veul$ and the elements of $\Veul_{A_n}$ are contained in $A_{n+1}$ by our assumption 
on $\Veul$. Clearly $W$ is open. Let $x \in X \setminus W$ be fixed for the moment. 
We are going to choose a suitable (small) 
neighborhood $U_x$ of $x$. 
By Lemma~\ref{neighborhood1} we may pick a neighborhood $W_x \subset X\setminus A_n$ of $x$ 
such that $W_x\subset V_J$ whenever $W_x \cap S_J$ is non-empty. 
In addition, 
after possibly shrinking $W_x$ a bit, we may assume that $W_x$ is contained in 
some element of $\cS$. If for all $|J|$ of dimension $k+1$ the sets $W_x$ and $S_J$ are disjoint, $W_x$ is a 
suitable neighborhood of $x$ for our purpose and we set $U_x\coloneqq W_x$. 
Otherwise consider for every $|J|$ of dimension 
$k+1$ with $W_x \cap S_J \not = \emptyset$ the restriction of $g$ to 
$$Q_J=\{0\}\times |J| \cup [0,1] \times \partial |J| \subset [0,1] \times N|\Veul|.$$ 
From $W_x \cap S_J \not = \emptyset$ we get $W_x\subset V_J$ by the choice of $W_x$. 
Therefore $r_{W_x} \circ g|Q_J$ is well defined and continuous. Since the 
pair $\Phi \hookrightarrow \Psi$ is a local weak homotopy equivalence there 
is a neighborhood $U_{J,x} \subset W_x$ of $x$ such that the map $r_{U_{J,x}} \circ g|Q_J$ 
extends to $[0,1] \times |J|$ with values in $\Psi(U_{J,x})$ and the restriction to $\{1\} \times |J|$ is 
mapped to $\Phi(U_{J,x})$. Since $x$ is contained in at most finitely many $S_j \in \Seul$, $x$ is 
also contained in at most finitely many $S_J$ with $\dim |J|=k+1$. Therefore we can replace every 
$U_{J,x}$ by the finite intersection and hence open neighborhood 
\begin{align*}
U_x=\bigcap_{x \in S_J} U_{J,x} \subset W_x
\end{align*} 
of $x$. We summarize that if the dimension of $|J|$ is $k+1$ and $U_x \cap S_J \not = \emptyset$, 
then $U_x \subset V_J$ and $r_{U_x} \circ g|Q_J$ extends to a map on $[0,1] \times |J|$ with 
values in $\Psi(U_x)$ such that $\{1\} \times |J|$ is mapped to $\Phi(U_x)$. Moreover, simply by shrinking 
$U_x$ more, we may assume that each $U_x$ is contained in some $A_{i+1} \setminus A_{i-1}$ for suitable 
$i \geq n+1$. By Lemma~\ref{cover1} 
there is a locally finite open family $\Aeul$ which refines $\{U_x\colon  x \in X\setminus W\}$, covers $X\setminus W$ 
and such that the order of $\Aeul_{A_{n+1}}$ is at most $\dim{A_{n+2}} +1$. Define 
$$\Ueul\coloneqq \Veul_{A_{n-1}} \cup (\Seul_{A_n}\setminus \Seul_{A_{n-1}}) \cup \Aeul$$
and let us check that $\Ueul$ satisfies all desired properties: $\Ueul$ covers $X$ because 
$\Aeul$ covers $X\setminus W$, and $\Seul_{A_n}$ covers $W$, hence 
$\Veul_{A_{n-1}} \cup (\Seul_{A_n}\setminus \Seul_{A_{n-1}})$ covers $W$ too. By construction we 
have $\Ueul_{A_{n-1}}=\Veul_{A_{n-1}}$. Every element of $\Ueul$ is 
contained in some $A_{i+1}\setminus A_{i-1}$ for suitable $i \in \N$ since this is the case for the elements 
of $\Aeul$, $\Veul$ and $\Seul$. Moreover, since the elements of $\Aeul$ do not meet $A_n$, we have $\ord{\Ueul_{A_n}}\leq \ord{\Veul_{A_n}}$ and
$$\ord(\Ueul_{A_{n+1}})\leq \ord(\Veul_{A_{n}}) + \ord(\Aeul_{A_{n+1}})\leq \ord(\Veul_{A_{n}}) + \dim A_{n+2}+1.$$ 
That $\Ueul$ refines $\Veul$ is clear since the elements of $\Aeul$ refine $\Seul$. 
Let us define the $k+1$-skeleton $h$. We have the indexed cover $\Veul=\{V_j\}_{j \in \Jeul}$ and let us index $\Ueul=\{U_i\}_{i \in \Ieul}$ such that 
$\Ieul \cap \Jeul$ are the indices corresponding to $\Veul_{A_n}$, and such that for $i \in \Ieul \cap \Jeul$ we have $U_i=S_i$ 
if $U_i \in \Seul_{A_n} \setminus \Seul_{A_{n-1}}$ and $U_i=V_i$ if $U_i \in \Veul_{A_{n-1}}$. Let $\{e_i\colon  i \in \Ieul \cup \Jeul\}$ denote linearly independent vectors, 
let $E$ be the real vector space spanned by $\{e_i\colon  i \in \Ieul\}$ and $E'$ the one 
spanned by $\{e_j\colon  j \in \Jeul\}$. We have $N|\Ueul|\subset E$ and $N|\Veul|\subset E'$. 
Pick for each $i \in \Ieul \setminus \Jeul$ an index $j(i) \in \Jeul$ 
such that $U_i\subset S_{j(i)} \subset V_{j(i)}$, and set $j(i)=i$ for $i \in \Ieul \cap \Jeul$.
Moreover let $\lambda\colon  N|\Ueul| \to N|\Veul|$ be the 
restriction of the linear map $E \to E'$ given by $e_i \mapsto e_{j(i)}$. From $U_i \subset V_{j(i)}$ it follows that 
the homotopy of $k$-skeletons

\begin{align*}
h\colon  [0,1] \times N_k|\Ueul| \to \Psi, \ \ h(t,x)\coloneqq g(t,\lambda(x))
\end{align*} 
is well-defined. Moreover $h_0$ is still the sectionally constant complex given by $f$ and $h_1$ has values in $\Phi$. 
The definition of $\lambda$ yields $h|[0,1] \times N|\Ueul_{A_{n-1}}|=g|[0,1] \times N|\Veul_{A_{n-1}}|$. 
We are left to extend $h$ to a homotopy of $k+1$-skeletons such that 
the time-$1$ map has values in $\Phi$.
To do this we are well prepared. Let 
$I \in N|\Ueul|$ be of length $k+1$. By assumption $\ord \Veul_{A_n}$ is smaller or equal to $k$, 
hence $I$ contains an index $i_0$ with $U_{i_0} \in \Aeul$ by the definition of $\Ueul$. 
Set $|J(I)|\coloneqq \lambda(|I|)$ and let $J(I)$ be the corresponding element of the nerve $N(\Veul)$. 
In the case where $\dim |J(I)| \leq k$, this extension exists since 
$U_I\subset V_{J(I)}$ and $g$ is defined on $N_k(\Veul)$ by assumption. Otherwise choose 
$U_x$ with $U_{i_0} \subset U_x$. 
Since $U_i \subset S_{j(i)}$ for all $i \in I$, we have $U_I \subset S_{J(I)}$. 
This implies
$$U_x \cap S_{J(I)} \supset U_I \cap S_{J(I)}=U_I \not = \emptyset,$$ 
hence $U_x \subset V_{J(I)}$ and $r_{U_x} \circ g|Q_{J(I)}$ extends to a continuous map 
$h_I\colon  [0,1] \times |J(I)| \to \Psi(U_x)$ with the desired properties by construction. 
Since $U_I\subset U_x$, we get an extension of $h$ to $[0,1] \times |I|$ by composing the restriction
$$\lambda\colon  [0,1] \times |I| \to [0,1] \times |J(I)|, \ \ (t,x) \mapsto (t,\lambda(x))$$ with $h_I$. 
This finishes the proof.
\end{proof}

\begin{proof}[Proof of Proposition~\ref{comp}]
First construct a suitable $0$-skeleton 
so that we can apply Lemma~\ref{inductive step} (with $n=1$) to that $0$-skeleton. 
As in the proof of Lemma~\ref{inductive step}, this depends on $\Phi \hookrightarrow \Psi$ 
being a local weak homotopy equivalence. In the following 
$\Veul=\Veul_n$ denotes the cover of $X$ corresponding to a homotopy 
of skeletons obtained in the $n$-th step. We repress the subscript $n$ 
to avoid ugly notations. In particular $\Veul$ may change with every step. 
Apply Lemma~\ref{inductive step} (with $n=1$) in a first step inductively $\dim A_3+1$ 
times to get a homotopy of $\dim A_3+1$-skeletons whose restriction 
to $[0,1]\times N|\Veul_{A_2}|$ is already a homotopy of complexes. Then apply 
Lemma~\ref{inductive step} (with $n=2$) in a second step for 
$\dim A_4+1$ more times to obtain a homotopy of $\dim A_3 +\dim A_4 +2$-skeletons which is already 
a homotopy of complexes when restricted to $[0,1]\times N|\Veul_{A_3}|$. In the third such step we get a homotopy of 
$\dim A_3 +\dim A_4 + \dim A_5+3$-skeletons which is a homotopy of complexes on $[0,1] \times N|\Veul_{A_4}|$ and 
the restriction to $[0,1] \times N|\Veul_{A_2}|$ (including the cover $\Veul_{A_2}$) 
has not been changed in this step. In the $n$-th such step we get a homotopy of skeletons 
which is a homotopy of complexes if restricted to $[0,1] \times N|\Veul_{A_{n+1}}|$ and the restriction 
to $[0,1] \times N|\Veul_{A_{n-1}}|$ (and $\Veul_{A_{n-1}}$) has not been changed since the last step. Moreover we get that 
the elements of $\Veul$ are contained in $A_{i+1}\setminus A_{i-1}$ for suitable $i \in \N$. Set $\Ueul_{n-1}\coloneqq \Veul_{A_{n-1}}$ and let $h_{n-1}$ 
be the restriction to $[0,1] \times N|\Veul_{A_{n-1}}|$ of the homotopy obtained in the $n$-th step for $n\geq 3$. We get 
that $\Ueul_n$ covers $A_n$, that $\Ueul_n \subset  \Ueul_{n+1}$, and 
that $h_{n+1}|[0,1] \times N|\Ueul_n|=h_n$ for $n\geq 2$. Moreover we get that 
the elements of $\Ueul_n$ are contained in $A_{i+1}\setminus A_{i-1}$ for suitable $i \in \N$, 
which guarantees that the union $\Ueul$ of all $\Ueul_n$, $n\geq 2$ is locally finite and that $N|\Ueul|$ is equal to the union of all $N|\Ueul_n|$, $n\geq 2$. 
Now 
$$h\colon [0,1] \times N|\Ueul| \to \Psi, \ \ h(t,x)\coloneqq h_n(t,x) \text{ for some $n$ with $x \in N|\Ueul_n|$}$$ 
is the desired homotopy of complexes. This finishes the proof.
\end{proof}

\subsection{Gluing sections} In this subsection we prove 
\begin{proposition}
\label{p2}
Let $\Aeul=(A_n)_{n \in \N}$ be an ordered cover of $X$ and let $\Phi$ be 
a sheaf of topological spaces on $X$. Assume that 
\begin{enumerate}
\item $\Phi$ is complete metric and $\Aeul$ is weakly flexible for $\Phi$, or 
\item $\Aeul$ is ordered flexible for $\Phi$. 
\end{enumerate}
Then any complex $f\colon N|\Aeul| \to \Phi$ is homotopic through a homotopy 
of complexes $f_t\colon N|\Aeul| \to \Phi$ to a sectionally constant complex.
\end{proposition}

The proof is based on work published in~\cite{FP}, 
where the task is done for the sheaf of holomorphic sections of elliptic submersions. 
This was developed in the thesis of Prezelj~\cite{Prezelj thesis}, see also~\cite{Francs book}.
In the following the symbol $\Phi$ denotes always a sheaf of topological spaces on a given space $X$.

\begin{lemma}
\label{extending homotopy}
Let $f\colon N|\Aeul| \to \Phi$ a complex and 
$g_t\colon N|\Beul| \to \Phi$ a homotopy of 
complexes with $g_0=f| N|\Beul|$ for some $\Beul\subset \Aeul$. Then 
there is a homotopy of complexes $f_t\colon N|\Aeul| \to \Phi$ with $f_0=f$ and 
$f_t|N|\Beul|=g_t$.
\end{lemma}

\begin{proof}
Note that we can write an arbitrary element of $N|\Aeul|$ uniquely as $sx+(1-s)y$, 
where $x \in N|\Beul|$, $y \in N|\Aeul\setminus \Beul|$ and $s \in [0,1]$. Define the 
extension $f_t$ of $g_t$ for $sx+(1-s)y\in N|\Aeul|$ by
\begin{align*}
f_t (sx+(1-s)y)=\begin{cases}
f((1+t)sx +(1-(1+t)s)y)							& \text{if} \ \ \ (1+t)s \leq 1 \\
g_{(1+t)s-1}(x)									& \text{if} \ \ \ (1+t)s > 1.
\end{cases}
\end{align*}
We get $f_0=f$, $f_t (x) = g_t (x) \ \text{and} \ f_t(y)=f(y)$, 
hence $f_t \vert N|\Beul| = g_t$ and $f_t \vert N|\Aeul\setminus \Beul|=f\vert N|\Aeul\setminus \Beul|$. 
For $I\cup J \in N(\Aeul)$ with $I \in N(\Beul)$ and 
$J \in N(\Aeul\setminus \Beul)$ we have $|I\cup J|=\{sx+(1-s)y\colon  s \in [0,1], x\in |I|, y \in |J|\}$. 
On the set $S\subset [0,1] \times |I\cup J|$ given by $(1+t)s\leq1$, $f_t$ is given by the 
composition of a continuous map $S\to  |I\cup J|$ and 
$f$, hence inherits the requested properties from $f$, whereas on 
the set $S'\subset [0,1] \times |I\cup J|$ given by $(1+t)s>1$, $f_t$ is given by the 
composition of a continuous map $S'\to  [0,1] \times |I|$ and 
$g_t$, hence the requested properties follow from those of $g_t$ and the fact that $I\subset I\cup J$. This implies together with 
$f_t (sx+(1-s)y)=f(x)=g_0(x)$ for $(1+t)s=1$ that $f_t$ is indeed a homotopy of complexes.
\end{proof}

\begin{lemma}
\label{main modification}
Let $f\colon  N|\Aeul| \to \Phi$ be a complex over a weakly or ordered flexible string 
$\Aeul=(A_1, A_2, \ldots, A_n)$ for $\Phi$ in $X$, $n\geq 2$, and set 
$\Beul=(A_1, \ldots, A_{n-1})$ and $A=A_1\cup \cdots \cup A_{n-1}$. Then

\begin{enumerate}
\item There is a homotopy of complexes $f_t$ over $\Aeul$ connecting $f_0=f$ to 
a sectionally constant complex $f_1$. 
\item If $f|N|\Beul|$ is sectionally constant, then $f_t$ can be chosen such that $f_t|N|\Beul|$ is sectionally 
constant too; and such that $r_{A^\circ}\circ g_t$ stays in a prescribed neighborhood of $r_{A^\circ} \circ g_0$, where 
$g_t \in \Phi(U)$, $t \in [0,1]$ for some neighborhood $U$ of $A$ denotes the homotopy 
induced by $f_t|N|\Beul|$.
\item If $\Aeul$ is ordered flexible and $f|N|\Beul|$ is sectionally 
constant we can strengthen (2) to $r_{A^\circ} \circ g_t$ being independent of $t$.
\end{enumerate}
\end{lemma}

\begin{proof}
We proceed by induction on $n$. For $n=2$ the statement is trivial since 
$(A_1, A_2)$ is a weakly (resp.~ordered) flexible pair by assumption. Suppose statement (1)
is true for some $n$ with $n-1 \geq 2$ and let $f\colon N|\Aeul| \to \Phi$ be a complex 
over $\Aeul$. Since $\Beul=(A_1, \ldots, A_{n-1})$ is a weakly flexible string of length 
$n-1$ for $\Phi$, we find by the inductive assumption a homotopy of complexes 
$\tilde f_t\colon  N|\Beul| \to \Phi$ connecting $\tilde f_0=f|N|\Beul|$ to a sectionally constant complex $\tilde f_1$. 
By Lemma~\ref{extending homotopy} 
there is a homotopy of complexes $f_t$ that extends $\tilde f_t$ to a homotopy of complexes over $\Aeul$ such that $f_0=f$, 
and $f_1|N|\Beul|=\tilde f_1$ is sectionally constant. 
Set $\Ceul=(A_1\cap A_n, \ldots, A_{n-1} \cap A_n)$, define $\lambda\colon [0,1] \times N|\Ceul| \to N|\Aeul|$ by $\lambda(s,e)=(1-s)e+se_n$ 
and note that $\lambda$ maps $[0,1] \times |I|$ for $I \in N(\Ceul)$ to $|I\cup \{n\}|$. Therefore $f \circ \lambda$ is a homotopy of complexes 
over $\Ceul$. Moreover $f \circ \lambda$ is sectionally constant when restricted to $s=\{0,1\}$. In particular we may view 
$f \circ \lambda$ as a complex over $\Ceul$ for the sheaf $\Phi_\alpha$, where $\alpha\colon \{0,1\} \to \Phi(V)$ is the map 
induced by $f\circ \lambda$ restricted to the set given by $s\in \{0,1\}$ for a 
sufficiently small neighborhood $V$ of the union of the elements of $\Ceul$. 
Since $\Ceul$ is weakly (resp.~ordered) flexible for $\Phi$, $\Ceul$ is likewise 
weakly (resp.~ordered) flexible for $\Phi_\alpha$ (see Lemma~\ref{parametric sheaves}). This implies by 
our inductive assumption that there is a homotopy $h_t$ connecting the 
complex given by $f \circ \lambda\colon  N|\Ceul| \to \Phi_\alpha $ to a sectionally constant one. 
Let us consider the homotopy $h_t\colon  N|\Ceul| \to \Phi_\alpha$, $t \in [0,1]$ as a 
homotopy $h_t\colon  [0,1] \times N|\Ceul| \to \Phi, \ (t,s,e') \mapsto h_t(s,e')$, 
which is sectionally constant and independent of $t$ for $s \in \{0,1\}$ and set
\begin{align*}
f_t(e)=
\begin{cases}
h_t\circ \lambda^{-1} (e),   	&\textit{ if } e\in \lambda((0,1) \times N|\Ceul|)=N|\Aeul|\setminus (N|\Beul| \cup \{e_n\}), \\
f(e), 						&\text{ otherwise. }
\end{cases}
\end{align*} 
Note that $f_0=f$. Moreover $f_t$ defines a homotopy of complexes over $\Aeul$ 
since $h_t$ is independent of $t\in [0,1]$ 
when restricted to $\{0,1\} \times N|\Ceul|$. We have $f_t|N|\Beul|=f|N|\Beul|$ by 
definition, and since $h_1$ is sectionally constant, $f_1$ yields a complex $f'$ 
over the pair $(A_1\cup \cdots \cup A_{n-1}, A_n)$. 
We are now in the situation of the inductive start and 
get a homotopy $f'_t\colon [0,1] \to \Phi$, $t \in [0,1]$ which connects $f'$ to a sectionally constant complex 
$f'_1$ and satisfies the approximation property $(2)$ in the weakly flexible 
case and the interpolation property $(3)$ in the ordered flexible case. 
The homotopy $f'_t$ yields the desired homotopy of complexes $f_t$, which can be seen 
explicitly by setting $f_t((1-s)e+se_n)\coloneqq f'_t(s)$ for $(1-s)e+se_n\in N|\Aeul|$. This finishes 
the proof of $(1)$. For the proof of $(2)$ and $(3)$ note that $h_t|N|\Beul|$ was independent of 
$t$ and that the homotopy of complexes $f_t$ from the last step satisfies the approximation resp.~interpolation 
property in question since $f'_t$ does. This finishes the proof.
\end{proof}

\begin{lemma}
\label{exhaustion}
Let $\Aeul=(A_n)_{n \in \N}$ be a countable ordered locally finite cover of $X$ by compacts. 
Then there is a subsequence $n_i$, $i \in \N$ of $\N$ such 
that $K_i\coloneqq A_1 \cup \cdots \cup A_{n_i}$ defines an exhaustion 
$K_1 \subset  K_2 \subset  K_3 \subset  \cdots$ of $X$.
\end{lemma}

\begin{proof}
It suffices to show that for $n \in \N$ the union
$K\coloneqq A_1 \cup \cdots \cup A_n$ admits an open neighborhood $U\supset K$ 
which is contained in a finite union of elements of $\Aeul$. Pick for each $p \in K$ a neighborhood 
$U_p$ of $p$ which meets at most finitely many elements of $\Aeul$ 
and then a finite subfamily $\{U_1, \ldots, U_l\}\subset \{U_p\colon  p \in K\}$ which covers $K$. 
The union $U\coloneqq U_1 \cup \cdots \cup U_l$ is an open neighborhood of $K$ which meets at most 
finitely many elements of $\Aeul$. Since $\Aeul$ covers $X$ the finite subfamily $\Aeul_U$ of elements which meet $U$ covers $U$. 
This finishes the proof.
\end{proof}

\begin{proof}[Proof of Proposition~\ref{p2}]
We first consider assumption (1), that is the case where $\Aeul$ is weakly flexible. Let $\epsilon>0$, set 
$\Aeul_n=(A_1, A_2, \ldots, A_n)$ and $U_n=(A_1\cup A_2 \cup \cdots \cup A_n)^\circ$ and denote the metric 
on $\Phi(U_n)$ by $d_{U_n}$. Applying Lemma~\ref{main modification} (2) and Lemma~\ref{extending homotopy} 
inductively ensures the existence of homotopies of complexes $f^n_t $, $n\in \N$ satisfying (i) $f^1_0=f$, $f^{n+1}_0=f^n_1$, 
(ii) $f^n_t|N|\Aeul_n|$ is sectionally constant, and (iii) the homotopy of maps $g^n_t$ induced by $f^n_t|N|\Aeul_n|$ defined 
on a neighborhood of $A_1 \cup \cdots \cup A_n$ satisfies 
$d_{U_m}(r_{U_m} \circ g^n_s, r_{U_m} \circ g^n_t) < \epsilon/2^n$ for all 
$s,t \in [0,1]$, $m\leq n$ ((iii) can be achieved since the restrictions of $\Phi$ are continuous and 
since there are only finitely many $m\leq n$ if $n \in \N$ is fixed). Moreover, by Lemma~\ref{exhaustion} there is for every $n \in \N$ 
some $m>n$ with $A_n\subset U_m$. Since $f^k||N|\Aeul_m|$ is sectionally constant for $k\geq m$ we may assume that the 
open neighborhood $\Ueul^k=(U_n^k)_{n \in \N}$ of $\Aeul=(A_n)_{n \in \N}$ corresponding to $f^k_t$ is such that $U_n^k \supset U_m$ for 
$k\geq m$. Write the half open unit interval as
\begin{align*}
\bigcup_{n \in \N} [t_n, t_{n+1}] =[0,1), \ \ \text{where} \ \ t_1 \coloneqq 0 \ \ \text{and} \ \ t_{n+1}\coloneqq \sum_{i=1}^n 1/2^{i} \ \ \text{for} \ \ n>0
\end{align*}
and rescale the homotopy parameter of $f^n_t$ to $[t_n,t_{n+1}]$. 
Let us denote the composition of the homotopies $f^1_t, f^2_t, f^3_t, \ldots$ 
with these rescaled parameters by $f_t$. Since $A_n \subset U_m \subset U_n^k$ for 
sufficiently large $m>n$ and all $k\geq m$, the intersection $\cap_{k \in \N} U_n^k$ 
contains an open neighborhood of $A_n$. This guarantees that the infinite composition $f_t$ of homotopies of 
complexes is again a homotopy of complexes with parameter $t \in [0,1)$.
We extend $f_t$ to $t=1$ by passing to the limit. 
For a fixed $I\in N(\Aeul)$ let $m$ be such that $A_I \subset U_m$, hence $r_{U_m} \circ f_t||I|$ 
is well defined and continuous for $t\geq t_m$; and set 
$f_1(e)=\lim_{t\to 1} r_{U_m}\circ f_t(e)$ for $e \in |I|$. This limit exists by (iii) and 
the fact that $\Phi(U_m)$ is a complete metric space by assumption. Moreover, by (iii) we have 
uniform convergence of $r_{U_m}\circ f_t||I|$ to $f_1||I|$, and hence 
$f_1||I|$ is a continuous extension of $r_{U_m}\circ f_t||I|$, $t\geq t_m$ 
to $[t_m,1]\times |I|$. $f_1$ is sectionally constant since $f_t|N|\Aeul_n|$ 
is sectionally constant for $t\geq t_n$ and all $n\in \N$. In particular $f_t$ is a 
homotopy of complexes connecting $f=f_0$ to a sectionally constant complex $f_1$. 
The case where $\Aeul$ is ordered flexible is much easier since instead of (iii) we can achieve 
the strengthening (iii'): The homotopy of maps $g^n_t$ induced by $f^n_t|N|\Aeul_n|$ defined on a neighborhood 
of $A_1 \cup \cdots \cup A_n$ satisfies $r_{U_m} \circ g^n_s= r_{U_m} \circ g^n_t$ 
for all $s,t \in [0,1]$, $m\leq n$. As an effect of this strengthening, passing to the limit 
does not require the sets of local sections to be complete metric since the 
homotopy is independent of $t$ when restricted to a fixed compact and sufficiently big $t<1$. 
This finishes the proof.
\end{proof}

\subsection{The proofs of Theorem~\ref{t1} and \ref{t3}} 
\begin{proof}[Proof of Theorem~\ref{t3}]
Let $\Phi \hookrightarrow \Psi$ be the inclusion of sheaves on $X$ 
corresponding to the assumptions of Theorem~\ref{t3}.
First we prove a non-parametric version. For given $f \in \Psi(X)$ 
Proposition~\ref{comp} yields a homotopy of complexes $g_t$, $t\in [0,1/2]$ over an open cover $\Ueul$ with values in 
$\Psi$ such that $g_0$ is the sectionally constant complex given by $f$ and $g_{1/2}$ has values in $\Phi$. 
By Remark~\ref{subcover}, since $\Phi$ is either weakly flexible or ordered flexible, we may exchange the cover $\Ueul$ corresponding to $g_t$ 
by a weakly flexible resp.~ordered flexible cover for $\Phi$. Now Proposition~\ref{p2} yields a homotopy of complexes with values in $\Phi$ connecting $g_{1/2}$ to 
a sectionally constant complex $g_1$. This shows in particular 
that $\Phi(X)\not = \emptyset$ if $\Psi(X)\not = \emptyset$ under the weaker assumptions stated in Theorem~\ref{t3}.
Including the assumption that $\Psi$ is likewise complete metric weakly flexible (or ordered flexible), note that 
the $\Psi$-valued homotopy of complexes $g_t$ is 
a sectionally constant complex when restricted to $t=0$ and $t=1$, hence defines a map 
$$\beta\colon   \{0,1\} \to \Psi(X)$$ with $\beta(0)=f$ given by $g_0$ and $\beta(1) \in \Phi(X)$ given by $g_1$. In particular 
$g_t$ yields a complex $g$ with values in $\Psi_\beta$. Since $\Psi_\beta$ is complete metric weakly flexible (resp.~ordered flexible) if 
$\Psi$ is (see Lemma~\ref{parametric sheaves}), we may assume that the cover corresponding to $g$ 
is weakly flexible (resp.~ordered flexible) for $\Psi_\beta$ by Lemma~\ref{faithful neighborhood} and Remark~\ref{subcover}.
Now Proposition~\ref{p2} yields a global section $h \in \Psi_\beta(X)$, which is a path with values in $\Psi(X)$ connecting $f$ to an element of $\Phi(X)$. 
This finishes the proof of the non-parametric version.
To pass to the parametric version, note that 
$\Phi(X) \hookrightarrow \Psi(X)$ 
is a weak homotopy equivalence if and only if for every $\alpha\colon \partial \D \to \Phi(X)$ and 
every $f \in \Psi_\alpha(X)$ there is a path in $\Psi_\alpha(X)$ connecting $f$ to an element of $\Phi_\alpha(X)$. 
Therefore Lemma~\ref{parametric sheaves} reduces the proof to the proved non-parametric version. This finishes the proof. 
\end{proof} \noindent
To prove Theorem~\ref{t1} let us define $\cC$-strings and $\cC$-covers in terms of 
$\cC$-pairs analogous to how we defined weakly flexible strings 
and weakly flexible covers in terms of weakly flexible pairs. That is

\begin{definition}
\label{C-cover}
Let $X$ be a complex space. A \textit{$\cC$-string} of length $n\geq 2$ is recursively 
defined as a finite sequence $(A_1, A_2,A_3, \ldots, A_n)$ of subsets of $X$ such that 
\begin{enumerate}
\item $(A_1 \cup \cdots \cup A_{n-1} , A_n)$ is a $\cC$-pair, and if $n\geq 3$, then
\item $(A_1, \ldots, A_{n-1})$ and $(A_1\cap A_n, \ldots , A_{n-1}\cap A_n)$ are $\cC$-strings. 
\end{enumerate}
A \textit{$\cC$-cover} is a locally finite cover $(A_n)_{n \in \N}$ of $X$ 
such that for every $n\in \N$  $(A_1, A_2, \ldots, A_n)$ is a $\cC$-string. 
\end{definition} \noindent
The crucial fact is that every Stein space admits arbitrarily fine $\cC$-covers, a 
result which relies on \textit{Grauert's bump method}. 
In the proof of an Oka principle this technique was initially applied by Henkin and Leiterer in \cite{Henkin und Leiterer}. 
A good and modern formulation of the required tool, 
which includes the case where $X$ is singular, is formulated 
in \cite{Francs book}, p.~294, a reference for 

\begin{proposition} 
\label{Cover}
Let $X$ be a second countable reduced Stein space and $\Ueul$ an open cover of $X$. 
Then there is a $\cC$-cover $(A_n)_{n \in \N}$ which refines $\Ueul$. 
\end{proposition} 
\noindent
For Stein spaces Proposition~\ref{Cover}  implies the following 
sufficient assumption for weak flexibility.

\begin{lemma}
\label{weak}
Let $\Phi$ be a sheaf of topological spaces on a second countable reduced Stein space $X$. 
Assume every point $p \in X$ has a neighborhood $U$ such that every 
$\cC$-pair $(A,B)$ with $B\subset U$ is weakly flexible for $\Phi$. Then $\Phi$ 
is weakly flexible. The analogous statement holds if weak flexibility is replaced by ordered flexibility 
in the assumption and the conclusion.
\end{lemma}

\begin{proof}
Let $\Ueul$ be an open cover of $X$. Pick for every point $p \in X$ a neighborhood $U_p$ such that each $\cC$-pair 
$(A,B)$ with $B\subset U_p$ is weakly flexible. By Proposition~\ref{Cover} there is a $\cC$-cover 
$\Aeul=(A_n)_{n \in \N}$ which refines $\{U_p \cap U\colon  p \in X, U \in \Ueul\}$. Clearly $\Aeul$ refines $\Ueul$, hence we are left to 
show that $\Aeul$ is a weakly flexible cover for $\Phi$. Note that every $\cC$-pair $(A,B)$ emerging from $\Aeul$ by applying the recursion 
in Definition~\ref{C-cover} satisfies $B\subset A_n \subset U_p$ for suitable $n \in \N$ and $p \in X$ and is therefore a weakly flexible pair by 
assumption. Moreover the recursions in the definitions of $\cC$-covers and weakly flexible covers are the same up to replacing 
every occurence of $\cC$ by \textit{weakly flexible}. From these two facts the result follows immediately.
\end{proof}

\begin{proof}[Proof of Theorem~\ref{t1}]
Sheaves satisfying assumption (1) resp.~(2) in Theorem~\ref{t1} are weakly resp.~ordered flexible by Lemma~\ref{weak}. 
Theorem~\ref{t1} is therefore a special case of Theorem~\ref{t3}.
\end{proof}

\appendix

\section{Applying Theorem~\ref{t1} in Oka theory}
In this appendix we give references for the proofs of the assumptions of Theorem~\ref{t1} in the settings of the Oka principles cited in the introduction. 
The intention is to give the reader a hint where the analytic challenges providing the assumptions of Theorem~\ref{t1} are tackled in the original work. 
In some cases the cited work needs some adjustments, which will be pointed out. These adjustments were part of the author's thesis~\cite{Diss}, but cost 
too many lines to be included here. 

\subsection{Weak flexibility}
The proof of the weak flexibility of a $\cC$-pair $(A,B)$ with respect to a given complex analytic sheaf is usually proved in two steps:
\begin{enumerate}
\item a parametric Runge approximation property, and  
\item a gluing property. 
\end{enumerate}
\vspace{2mm} \noindent
\textit{The Oka principle for elliptic submersions:} To show the weak flexibility assumptions of Theorem~\ref{t1} in this setting 
one has to show that if $h\colon Z \to X$ is 
a holomorphic submersion onto a reduced Stein space $X$ and $U\subset X$ is an open set such that the restriction 
$h\colon h^{-1}(U) \to U$ admits a dominating spray, then every $\cC$-pair $(A,B)$ with $B\subset U$ is weakly flexible for 
the sheaf of holomorphic sections of $h$. This was discovered by Gromov~\cite{elliptic bundles}.
Detailed proofs of Gromov's insight have been given by 
Forstneri\v c and Prezelj (see e.g.~\cite{FP}). A convenient source is~\cite{Francs book}: The weak flexibility of 
the pair $(A,B)$ follows from the Runge approximation property stated in Theorem 6.2.2, p.~284 and the gluing property stated in Proposition 6.7.2, p.~288. 
\vspace{2mm} \\
\textit{The Oka principle for principal $G$-bundles:} This Oka principle is a special case of the Oka principle for elliptic submersions. However, since 
the two remaining Oka principles build strongly on Cartan's exposition of Grauert's work~\cite{Cartan}, it makes sense to give references. The required 
Runge approximation property and gluing property which yield the weak flexibility of $\cC$-pairs follow in this case from 
a Runge approximation property and a splitting lemma in an associated sheaf of groups. These two key results are in Cartan's exposition of 
Grauert's work Proposition 1 and 2 (see~\cite{Cartan}, p.~109).
\vspace{2mm} \\ 
\textit{The Oka principle for admissible pairs of sheaves:} This Oka principle builds on Cartan's exposition of Grauert's work. The necessary extensions 
of Cartan's Proposition 1 and 2 from~\cite{Cartan} are Lemma 2 and 3 from Forster and Ramspott's work (see~\cite{Forster und Ramspott}, p.~271 and p.~273).
\vspace{2mm} \\ 
\textit{The Oka principle for equivariant isomorphisms:} This Oka principle builds likewise on Cartan's text~\cite{Cartan}. The necessary extensions of 
Cartan's Proposition 1 and 2 in the work of Kutzschebauch, L\'arusson and Schwarz are Proposition 10.2 and 10.3 (see~\cite{KLS}, p.~7293).
\vspace{2mm} \\
The key results from the last three Oka principles, i.e.\ from those Oka principles which build on Cartan's exposition of Grauert's work, need some adjustments 
to yield complete proofs of the weak flexibility of $\cC$-pairs. These adjustments can be found in~\cite{Diss}, Chapter 5 and 6.

\subsection{Local weak homotopy equivalences} The difficulty of the proof that a given inclusion of sheaves $\Phi \hookrightarrow \Psi$ 
is a local weak homotopy equivalence depends strongly on the setting. 
\vspace{2mm} \\ \textit{The Oka principle for elliptic submersions:} In this setting 
it suffices to show that if 
$h\colon Z \to X$ is a holomorphic submersion onto a reduced complex space and $\Phi \hookrightarrow \Psi$ is 
the inclusion of the sheaf of holomorphic sections to the sheaf of continuous sections of $h$, then $\Phi \hookrightarrow \Psi$ is 
a local weak homotopy equivalence. Gromov seems to have taken this result for granted in~\cite{elliptic bundles}. 
In the more detailed work~\cite{FP} local weak homotopy equivalences are not introduced. Instead, an analogue of our Proposition~\ref{comp} 
is stated in the special case of holomorphic submersions, namely Proposition 4.7. The validity of Proposition 4.7 in~\cite{FP} has been carefully checked in 
the thesis of Jasna Prezelj~\cite{Prezelj thesis}, which yields implicitly a proof of the fact that $\Phi \hookrightarrow \Psi$ is a local weak homotopy equivalence in the mentioned case. 
\vspace{2mm} 
\\ \textit{The Oka principle for principal $G$-bundles:} This is a special case of the above.\vspace{2mm} \\
\textit{The Oka principle for admissible pairs of sheaves:} In the work of 
Forster and Ramspott~\cite{Forster und Ramspott} there is a slight weakening of $\Phi \hookrightarrow \Psi$ being a local weak homotopy equivalence 
in the assumption, namely the homotopy property (PH) from Satz 1, p.~267. 
Using (PH), the fact that an inclusion of admissible pairs of sheaves $\Phi \hookrightarrow \Psi$ 
in the sense of Forster and Ramspott is a local weak homotopy equivalence 
is a corollary to Lemma 1, p.~269 in~\cite{Forster und Ramspott}. 
\vspace{2mm} \\ 
\textit{The Oka principle for equivariant isomorphisms:} In this setting 
it is hard to show that the given inclusion $\Phi \hookrightarrow \Psi$ is a local weak homotopy equivalence. 
Theorem 1.3, p.~7253 in~\cite{KLS} reduces the proof to the case where we have $X=Y$ for given Stein $G$-manifolds $X$ and $Y$, where $G$ is a 
complex reductive Lie group. Having this, one needs to extend some Lemmata from~\cite{KLS} in Section 3 and 5 to analogous parametric versions. 
These necessary adaptations are simple once Section 3 and 5 in~\cite{KLS} are understood. \vspace{2mm}

For more details that the inclusions $\Phi \hookrightarrow \Psi$ from the mentioned Oka principles are local weak homotopy equivalences 
see~\cite{Diss}, Chapter 4. 
\subsection{Sheaves of topological spaces} Equipping all sets of local sections from the sheaves 
corresponding to the Oka principles from the introduction with the compact open topology turns these 
into sheaves of topological spaces. This is used in all the quoted work. To see that there is no pathological behavior when dealing 
with parametric sheaves, see Lemma~\ref{correspondence}, a basic fact which is usually taken for granted in Oka theory. Recall that (complete) 
metric sheaves are defined as those sheaves of topological spaces whose sets of local sections are equipped with a 
(complete) metric which induces the topology. That the complex analytic sheaves from the mentioned Oka principles 
are complete metric depends on the following two facts.

\begin{fact}
\label{f1}
Let $X$ be a space which admits an exhaustion $K_1\subset K_2\subset K_3 \subset \cdots$ by compacts and $(Y,d)$ a complete metric space. 
Then
$$d_{C(X,Y)}(f,g)=\sum_{n\geq 1} \frac{1}{2^n} \frac{d_n(f,g)}{1+d_n(f,g)}, \ \ d_n(f,g)=\max_{x \in K_n} d(f(x),g(x))$$ for continuous maps $f,g\colon X\to Y$ 
defines a complete metric on the space of continuous maps $X \to Y$ which induces the compact open topology.
\end{fact}

\begin{fact}
\label{f2}
Let $X,Y$ be two locally connected locally compact second countable metric spaces. Then 
$d_H(f,g)=d_{C(X,Y)}(f,g)+d_{C(Y,X)}(f^{-1},g^{-1})$, where $d_{C(X,Y)}$ and $d_{C(Y,X)}$ are 
as in Fact~\ref{f1}, defines a complete metric on the set of homeomorphisms $X \to Y$ which induces 
the compact open topology.
\end{fact} 

Fact~\ref{f1} is well known, and it is easy to show that $d_H$ from Fact~\ref{f2} turns the set of homeomorphisms into a complete metric space. 
It is not obvious (and for locally disconnected topological spaces generally false) that the topology induced by $d_H$ is 
not finer than the compact open topology. The proof depends on the following

\begin{thm}[Arens~\cite{Arens}]
The homeomorphism group of a locally compact locally connected Hausdorff space equipped with the compact open topology is a topological group. 
\end{thm} 

To show that $d_H$ induces the compact open topology it suffices to show that $d_H(\cdot, g)$ is continuous with respect to the compact open topology for a fixed homeomorphism 
$g\colon X \to Y$; and this is the case if $I\colon \Homeo(Y,Z) \to \Homeo(Z,Y), \ I(f)=f^{-1}$ is continuous. The latter follows from the fact that $I$ is composition of the three continuous maps 
given by $f \mapsto g^{-1} \circ f$, $f \mapsto f^{-1}$ and $f \mapsto f \circ g^{-1}$, where the continuity of the second factor is due to Arens result. 

Fact~\ref{f1} and~\ref{f2} imply that the considered complex analytic sheaves from the mentioned Oka principles are complete metric in the following way:
It is known that the sets of local sections of the complex analytic sheaves from the mentioned Oka principles are closed subspaces either of a space 
of continuous maps $X \to Y$ for suitable $X$ and $Y$ or from the space of homeomorphisms $X \to Y$ for suitable $X$ and $Y$. 
In any case, Fact~\ref{f1} and~\ref{f2} imply that the sets of local sections are closed subsets of a complete metric space and 
hence complete if equipped with the suitable metric from Fact~\ref{f1} resp.~Fact~\ref{f2}. 

\subsection{Ordered flexibility}
Ordered flexibility is -- in the context of Oka theory -- most of times rather easy to show. Proofs shall be given elsewhere. 
Instead, we would like to discuss the examples addressed in Remark~\ref{remark ordered flex}. 
In most known Oka principles one looks at inclusions of sheaves $\Phi \hookrightarrow \Psi$, where $\Phi$ lives in 
the complex analytic category and $\Psi$ lives in the category of topological spaces. However, in recent advances (see e.g.~\cite{KLS, Leiterer}) 
one is forced to place $\Psi$ in the smooth category instead. In the smooth category -- opposed to the complex analytic and the topological 
category -- completeness is more delicate. One can turn e.g. the space of smooth functions $\C \to \C$ into a complete metric space by 
including all higher derivatives in the definition of the pseudometrics from Fact~\ref{f1}. However, the resulting topology is finer than the 
compact open topology. This is a disadvantage for proofs in Oka theory since Lemma~\ref{correspondence} does not apply anymore. A similar example 
emerges from~\cite{forum}, where one is forced to look at a sheaf $\Psi$ of continuous sections which are holomorphic in a neighborhood 
of some fixed subvariety. There seems to be no (natural) way to turn this sheaf into a complete metric sheaf without refining the 
compact open topology. In these examples one benefits from assumption (2) in Theorem~\ref{t1} as an alternative to 
assumption (1), since in (2) no completeness is asked.

\end{document}